\input ./style/arxiv-vmsta.cfg
\documentclass[numbers,compress,v1.0.1]{vmsta}
\usepackage{vtexbibtags}
\usepackage{authorquery}

\volume{5}
\issue{4}
\pubyear{2018}
\firstpage{471}
\lastpage{481}
\aid{VMSTA117}
\doi{10.15559/18-VMSTA117}
\articletype{research-article}

\startlocaldefs

\newcommand{\rrvert}{\vert}
\newcommand{\llvert}{\vert}
\allowdisplaybreaks

\newtheorem{thm}{Theorem}
\newtheorem{remark}{Remark}
\newtheorem{lemma}{Lemma}

\newtheorem{proposition}{Proposition}

\theoremstyle{definition}
\newtheorem{defin}{Definition}
\endlocaldefs

\begin{aqf}
\querytext{Q1}{Note: 'continuous in process' is an incomplete phrase. Please complete.}
\end{aqf}
\begin{document}

\begin{frontmatter}
\pretitle{Research Article}

\title{On generalized stochastic fractional integrals and related inequalities}

\author{\inits{H.}\fnms{H\"{u}seyin}~\snm{Budak}\thanksref{cor1}\ead[label=e1]{hsyn.budak@gmail.com}}
\thankstext[type=corresp,id=cor1]{Corresponding author.}
\author{\inits{M.Z.}\fnms{Mehmet Zeki}~\snm{Sarikaya}\ead[label=e2]{sarikayamz@gmail.com}}

\address{Department of Mathematics, Faculty of Science and Arts,
\institution{D\"{u}zce University}, Konuralp Campus, D\"{u}zce-\cny{Turkey}}

\markboth{H. Budak, M.Z. Sarikaya}{On generalized stochastic
fractional integrals and related inequalities}




\begin{abstract}
The generalized mean-square fractional
integrals $\mathcal{J}_{\rho,\lambda,u+;\omega}^{\sigma}$ and $\mathcal{J}_{\rho,\lambda,v-;\omega}^{\sigma}$ of the stochastic process $X$ are introduced.
Then, for Jensen-convex and strongly convex stochastic proceses, the generalized fractional Hermite--Hadamard inequality is establish
via generalized stochastic fractional integrals.
\end{abstract}

\received{\sday{29} \smonth{5} \syear{2018}}
\revised{\sday{13} \smonth{9} \syear{2018}}
\accepted{\sday{13} \smonth{9} \syear{2018}}
\publishedonline{\sday{24} \smonth{9} \syear{2018}}
\end{frontmatter}

\section{Introduction}

In 1980, Nikodem \cite{Nikodem} introduced convex stochastic processes and
investigated their regularity properties. In 1992, Skwronski \cite%
{Skowronski} obtained some further results on convex stochastic processes.

Let $ ( \varOmega,\mathcal{A},P ) $ be an arbitrary probability
space. A function $X:\varOmega\rightarrow\mathbb{R}$ is called a random
variable if it is $\mathcal{A}$-measurable. A function $X:I\times\varOmega
\rightarrow\mathbb{R}$, where $I\subset\mathbb{R}$ is an interval, is
called a stochastic process if for every $t\in I$ the function $X (
t,. ) $ is a random variable.

Recall that the stochastic process $X:I\times\varOmega\rightarrow\mathbb{R}$
is called

(\textit{i}) continuous in probability in interval $I$, if for
all $%
t_{0}\in I$ we have
\begin{equation*}
P{\textrm{-}}\lim_{t\rightarrow t_{0}}X ( t,. ) =X ( t_{0},. ) ,
\end{equation*}
where $P{\textrm{-}}\lim$ denotes the limit in probability.

(\textit{ii}) \textit{mean-square continuous} in the interval
$I$, if
for all $t_{0}\in I$
\begin{equation*}
\lim_{t\rightarrow t_{0}}E \bigl[ \bigl( X ( t ) -X ( t_{0} )
\bigr) ^{2} \bigr] =0,
\end{equation*}
where $E [ X ( t )  ] $ denotes the expectation value of
the random variable $X ( t,. ) $.

Obviously, \textit{mean-square }continuity implies continuity in
probability, but the converse implication is not true.

\begin{defin}
Suppose we are given a sequence $ \{ \Delta^{m} \} $ of
partitions, $\Delta^{m}= \{ a_{m,0},\ldots,\break a_{m,n_{m}} \} $.
We say
that the sequence $ \{ \Delta^{m} \} $ is a normal sequence of
partitions if the length of the greatest interval in the $n$-th partition
tends to zero, i.e.,
\begin{equation*}
\lim\limits
_{m\rightarrow\infty}\sup_{1\leq i\leq n_{m}} \llvert a_{m,i}-a_{m,i-1}
\rrvert =0.
\end{equation*}
\end{defin}

Now we would like to recall the concept of the mean-square integral.
For the
definition and basic properties see \cite{Sobczyk}.

Let $X:I\times\varOmega\rightarrow\mathbb{R}$ be a stochastic process
with $E [ X ( t ) ^{2} ] <\infty$ for all $t\in I$. Let
$ [a,b ] \subset I$, $a=t_{0}<t_{1}<t_{2}<\cdots<t_{n}=b$ be a partition
of $ [ a,b ] $ and $\varTheta_{k}\in [ t_{k-1},t_{k} ] $ for
all $k=1,\ldots,n$. A random variable $Y:\varOmega\rightarrow\mathbb{R}$ is
called the mean-square integral of the process $X$ on $ [ a,b] $,
if we have
\begin{equation*}
\lim_{n\rightarrow\infty}E \Biggl[ \Biggl( \sum_{k=1}^{n}X
( \varTheta _{k} ) ( t_{k}-t_{k-1} ) -Y \Biggr)
^{2} \Biggr] =0
\end{equation*}
for all normal sequences of partitions of the interval $ [ a,b
] $
and for all $\varTheta_{k}\in [ t_{k-1},t_{k} ] $, $k=1,\ldots,n$.
Then, we write
\begin{equation*}
Y ( \cdot ) =\int\limits
_{a}^{b}X ( s,\cdot ) ds\text{ (a.e.).}
\end{equation*}
For the existence of the mean-square integral it is enough to assume the
mean-square continuity of the stochastic process $X$.

Throughout the paper we will frequently use the monotonicity of the
mean-square integral. If $X ( t,\cdot ) \leq Y ( t,\cdot
 ) $ (a.e.) in some interval $ [ a,b ] $, then
\begin{equation*}
\int\limits
_{a}^{b}X ( t,\cdot ) dt\leq\int \limits
_{a}^{b}Y
( t,\cdot ) dt\text{ (a.e.).}
\end{equation*}
Of course, this inequality is an immediate consequence of the
definition of
the mean-square integral.

\begin{defin}
We say that a stochastic processes $X:I\times\varOmega\rightarrow\mathbb{R}$
is convex, if for all $\lambda\in [ 0,1 ] $ and $u,v\in I$ the
inequality
\begin{equation}
X \bigl( \lambda u+ ( 1-\lambda ) v,\cdot \bigr) \leq\lambda X ( u,\cdot ) + (
1-\lambda ) X ( v,\cdot ) \text{ \ \ (a.e.)} \label{H1}
\end{equation}
is satisfied. If the above inequality is assumed only for $\lambda
=\frac{1}{%
2}$, then the process $X$ is Jensen-convex or $\frac{1}{2}$-convex. A
stochastic process $X$ is concave if $ ( -X ) $ is convex. Some
interesting properties of convex and Jensen-convex processes are presented
in \cite{Nikodem,Sobczyk}.
\end{defin}

Now, we present some results proved by Kotrys \cite{Kotrys} about
Hermite--Hadamard inequality for convex stochastic processes.

\begin{lemma}
\label{LL} If $X:I\times\varOmega\rightarrow\mathbb{R}$ is a stochastic
process of the form $X ( t,\cdot ) =A ( \cdot )
t+B ( \cdot ) $, where $A,B:\varOmega\rightarrow\mathbb{R}$ are
random variables, such that $E [ A^{2} ] <\infty,E [ B^{2}
 ] <\infty$ and $ [ a,b ] \subset I$, then
\begin{equation*}
\int\limits
_{a}^{b}X ( t,\cdot ) dt=A ( \cdot )
\frac{%
b^{2}-a^{2}}{2}+B ( \cdot ) ( b-a ) \text{ (a.e.).}
\end{equation*}
\end{lemma}

\begin{proposition}
\label{p} Let $X:I\times\varOmega\rightarrow\mathbb{R}$ be a convex
stochastic process and $t_{0}\in intI$. Then there exists a random
variable $%
A:\varOmega\rightarrow\mathbb{R}$ such that $X$ is supported at $t_{0}$ by
the process $A ( \cdot )  ( t-t_{0} ) +X (
t_{0},\cdot ) $. That is
\begin{equation*}
X ( t,\cdot ) \geq A ( \cdot ) ( t-t_{0} ) +X ( t_{0},
\cdot ) \text{ (a.e.).}
\end{equation*}
for all $t\in I$.
\end{proposition}

\begin{thm}
\label{LL1} Let $X:I\times\varOmega\rightarrow\mathbb{R}$ be a Jensen-convex,
mean-square continuous in the interval $I$ stochastic process. Then for
any $%
u,v\in I$ we have
\begin{equation}
X \biggl( \frac{u+v}{2},\cdot \biggr) \leq\frac{1}{v-u}\int \limits
_{u}^{v}X
( t,\cdot ) dt\leq\frac{X ( u,\cdot
 )
+X ( v,\cdot ) }{2}\text{ (a.e.)} \label{0}
\end{equation}
\end{thm}

In \cite{kotrys2}, Kotrys introduced the concept of strongly convex
stochastic processes and investigated their properties.

\begin{defin}
Let $C:%
\varOmega
\rightarrow\mathbb{R}$ denote a positive random variable. The stochastic
process $X:I\times\varOmega\rightarrow\mathbb{R}$ is called strongly convex
with modulus $C(\cdot)>0$, if for all $\lambda\in
{}[
0,1]$ and $u,v\in I$ the inequality
\begin{equation*}
X\bigl(\lambda u+(1-\lambda)v,\cdot\bigr)\leq\lambda X(u,\cdot)+(1-\lambda )X(v,
\cdot)-C(\cdot)\lambda(1-\lambda) (u-v)^{2}\text{ \ \ \ \ \ \ a.e.}
\end{equation*}
is satisfied. If the above inequality is assumed only for $\lambda
=\frac{1}{%
2}$, then the process $X$ is strongly Jensen-convex with modulus
$C(\cdot)$.
\end{defin}

In \cite{hafiz}, Hafiz gave the following definition of stochastic
mean-square fractional integrals.

\begin{defin}
For the stochastic proces $X:I\times\varOmega\rightarrow\mathbb{R}$, the
concept of stochastic mean-square fractional integrals $I_{u+}^{\alpha}$
and $I_{v+}^{\alpha}$ of $X$ of order $\alpha>0$ is defined by
\begin{equation*}
I_{u+}^{\alpha} [ X ] (t)=\frac{1}{\varGamma(\alpha)}\int \limits
_{u}^{t}(t-s)^{\alpha-1}X(x,s)ds
\text{ \ \ }(a.e.),\text{ \ \ }t>u,
\end{equation*}
and
\begin{equation*}
I_{v-}^{\alpha} [ X ] (t)=\frac{1}{\varGamma(\alpha)}\int \limits
_{t}^{v}(s-t)^{\alpha-1}X(x,s)ds
\text{ \ \ }(a.e.),\text{ \ \ }t<v.
\end{equation*}
\end{defin}

Using this concept of stochastic mean-square fractional integrals $%
I_{a+}^{\alpha}$ and $I_{b+}^{\alpha}$, Agahi and Babakhani proved the
following Hermite--Hadamard type inequality for convex stochastic processes:

\begin{thm}
Let $X:I\times\varOmega\rightarrow\mathbb{R}$ be a Jensen-convex stochastic
process that is mean-square continuous in the interval $I$. Then for
any $%
u,v\in I$, the following Hermite--Hadamard inequality
\begin{equation}
X \biggl( \frac{u+v}{2},\cdot \biggr) \leq\frac{\varGamma(\alpha
+1)}{2 (
v-u ) ^{\alpha}} \bigl[
I_{u+}^{\alpha} [ X ] (v)+I_{v-}^{\alpha} [ X ] (u)
\bigr] \leq\frac{X (
u,\cdot
 ) +X ( v,\cdot ) }{2}\text{ (a.e.)} \label{E3}
\end{equation}
holds, where $\alpha>0$.
\end{thm}

For more information and recent developments on Hermite--Hadamard type
inequalities for stochastic process, please refer to \cite{barraez,budak,gonzalez,maden,materano,
Nikodem,
sarikaya,set,set2,tomar,tomar2}.

\section{Main results}

In tis section, we introduce the concept of the generalized mean-square
fractional integrals $\mathcal{J}_{\rho,\lambda,u+;\omega}^{\sigma
}$ and
$\mathcal{J}_{\rho,\lambda,v-;\omega}^{\sigma}$ of the stochastic
process $X$.

In \cite{raina}, Raina studied a class of functions defined formally by
\begin{equation}
\mathcal{F}_{\rho,\lambda}^{\sigma} ( x ) =\mathcal{F}_{\rho
,\lambda}^{\sigma ( 0 ) ,\sigma ( 1 ) ,\ldots}
( x ) =\sum_{k=0}^{\infty}\frac{\sigma ( k ) }{\varGamma
 (
\rho k+\lambda ) }x^{k}
\text{ \ \ \ } \bigl( \rho,\lambda >0; \llvert x \rrvert <\mathcal{R} \bigr) ,
\label{E4}
\end{equation}
where the cofficents $\sigma ( k ) $ $ ( k\in\mathbb
{N}_{0}=%
\mathbb{N\cup} \{ 0 \}  ) $ make a bounded sequence of
positive real numbers and \ $\mathcal{R}$\ is the set of real numbers. For
more information on the function (\ref{E4}), please refer to \cite{lua,
parmar}. With the help of (\ref{E4}), we give the following definition.

\begin{defin}
Let $X:I\times\varOmega\rightarrow\mathbb{R}$ be a stochastic process. The
generalized mean-square fractional integrals $\mathcal{J}_{\rho,\lambda
,a+;\omega}^{\alpha}$ and $\mathcal{J}_{\rho,\lambda,b-;\omega
}^{\alpha
}$ of $X$ are defined by
\begin{equation}
\mathcal{J}_{\rho,\lambda,u+;\omega}^{\sigma} [ X ] (x)=\int_{u}^{x}
( x-t ) ^{\lambda-1}\mathcal{F}_{\rho
,\lambda
}^{\sigma} \bigl[ \omega (
x-t ) ^{\rho} \bigr] X(t,\cdot )dt,\ \ \text{ (a.e.)}\ \ x>u,
\label{E5}
\end{equation}
and
\begin{equation}
\mathcal{J}_{\rho,\lambda,v-;\omega}^{\sigma} [ X ] (x)=\int_{x}^{v}
( t-x ) ^{\lambda-1}\mathcal{F}_{\rho
,\lambda
}^{\sigma} \bigl[ \omega (
t-x ) ^{\rho} \bigr] X(s,\cdot )dt,\ \ \ \text{(a.e.)}\ \ x<v,
\label{E6}
\end{equation}
where $\lambda,\rho>0,\omega\in\mathbb{R}$.
\end{defin}

Many useful generalized mean-square fractional integrals can be
obtained by
specializing the coefficients $\sigma(k)$. Here, we just point out that the
stochastic mean-square fractional integrals $I_{a+}^{\alpha}$ and $%
I_{b+}^{\alpha}$ can be established by coosing $\lambda=\alpha$,
$\sigma
(0)=1$ and $w=0$.

Now we present Hermite--Hadamard inequality for generalized mean-square
fractional integrals $\mathcal{J}_{\rho,\lambda,a+;\omega}^{\sigma
}$ and
$\mathcal{J}_{\rho,\lambda,b-;\omega}^{\sigma}$ of $X$.

\begin{thm}
\label{t1} Let $X:I\times\varOmega\rightarrow\mathbb{R}$ be a Jensen-convex
stochastic process that is mean-square continuous in the interval $I$. For
every $u,v\in I,\ u<v$, we have the following Hermite--Hadamard
inequality
\begin{align}
& X \biggl( \frac{u+v}{2},\cdot \biggr) \label{h0}
\\
&\quad\leq \frac{1}{2 ( v-u )
^{\lambda}\mathcal{F}_{\rho,\lambda+1}^{\sigma}
 [ \omega (
v-u ) ^{\rho} ] } \bigl[ \mathcal{J}_{\rho,\lambda,u+;\omega
}^{\sigma} [ X ]
(t)+\mathcal{J}_{\rho,\lambda,v-;\omega
}^{\sigma} [ X ] (t) \bigr] \notag
\\&\quad\leq\frac{X ( u,\cdot ) +X ( v,\cdot ) }{2}.\text { \ \ \ \ \ \ \ \ \ \ \ \ \ \ \ \ \ \ \ \ \ \ \ \
\ \ \ \ \ \ \ \ \ \ \ \ \ \ \ \ \ \ \ \ \ \ \ \ \ \ \ a.e.} \notag
\end{align}
\end{thm}

\begin{proof}
Since the process $X$ is mean-square continuous, it is continuous in
probability. Nikodem \cite{Nikodem} proved that every Jensen-convex and continuous
in probability stochastic process is convex. Since $X$ is convex,
then by Proposition 1, it has a supporting process at any point
$t_{0}\in
intI$. Let us take a support at $t_{0}=\frac{u+v}{2}$, then we have
\begin{equation}
X ( t,\cdot ) \geq A(\cdot) \biggl( t-\frac{u+v}{2} \biggr) +X \biggl(
\frac{u+v}{2},\cdot \biggr) .\text{ \ \ \ a.e.} \label{h1}
\end{equation}
Multiplying both sides of (\ref{h1}) by $ [  ( v-t )
^{\lambda
-1}\mathcal{F}_{\rho,\lambda}^{\sigma} [ \omega ( v-t )
^{\rho} ] + ( t-u ) ^{\lambda-1}\mathcal{F}_{\rho
,\lambda
}^{\sigma} [ \omega ( t-u ) ^{\rho} ]  ] $, then
integrating the resulting inequality with respect to $t$ over $[u,v]$, we
obtain
\begin{align}
&\int\limits
_{u}^{v} \bigl[ ( v-t ) ^{\lambda-1}
\mathcal{F}%
_{\rho,\lambda}^{\sigma} \bigl[ \omega ( v-t )
^{\rho
} \bigr] + ( t-u ) ^{\lambda-1}\mathcal{F}_{\rho,\lambda}^{\sigma
}
\bigl[ \omega ( t-u ) ^{\rho} \bigr] \bigr] X ( t,\cdot ) dt
\label{3}
\\
&\quad\geq A(\cdot)\int\limits
_{u}^{v} \bigl[ ( v-t ) ^{\lambda
-1}%
\mathcal{F}_{\rho,\lambda}^{\sigma} \bigl[ \omega ( v-t ) ^{\rho
}
\bigr] \notag\\
&\qquad+ ( t-u ) ^{\lambda-1}\mathcal{F}_{\rho,\lambda
}^{\sigma} \bigl[
\omega ( t-u ) ^{\rho} \bigr] \bigr] \biggl( t-%
\frac{u+v}{2} \biggr) dt \notag
\\
&\qquad+X \biggl( \frac{u+v}{2},\cdot \biggr) \int\limits
_{u}^{v} \bigl[
( v-t ) ^{\lambda-1}\mathcal{F}_{\rho,\lambda}^{\sigma} \bigl[ \omega (
v-t ) ^{\rho} \bigr] \notag\\
&\qquad+ ( t-u ) ^{\lambda
-1}\mathcal{F}%
_{\rho,\lambda}^{\sigma}
\bigl[ \omega ( t-u ) ^{\rho
} \bigr] %
 \bigr] dt \notag\\
&\quad=A(\cdot)\int\limits
_{u}^{v} \bigl[ t ( v-t ) ^{\lambda-1}%
\mathcal{F}_{\rho,\lambda}^{\sigma} \bigl[ \omega ( v-t ) ^{\rho
}
\bigr] +t ( t-u ) ^{\lambda-1}\mathcal{F}_{\rho,\lambda
}^{\sigma} \bigl[
\omega ( t-u ) ^{\rho} \bigr] \bigr] dt \notag
\\
&\qquad-A(\cdot)\frac{u+v}{2}\int\limits
_{u}^{v} \bigl[ ( v-t )
^{\lambda-1}\mathcal{F}_{\rho,\lambda}^{\sigma} \bigl[ \omega ( v-t )
^{\rho} \bigr] + ( t-u ) ^{\lambda-1}\mathcal {F}_{\rho
,\lambda}^{\sigma}
\bigl[ \omega ( t-u ) ^{\rho} \bigr] \bigr] dt \notag
\\
&\qquad+X \biggl( \frac{u+v}{2},\cdot \biggr) \int\limits
_{u}^{v} \bigl[
( v-t ) ^{\lambda-1}\mathcal{F}_{\rho,\lambda}^{\sigma} \bigl[ \omega (
v-t ) ^{\rho} \bigr]\notag\\
&\qquad + ( t-u ) ^{\lambda
-1}\mathcal{F}%
_{\rho,\lambda}^{\sigma}
\bigl[ \omega ( t-u ) ^{\rho
} \bigr] %
 \bigr] dt. \notag
\end{align}

Calculating the integrals, we have
\begin{align}
&\int\limits
_{u}^{v}t ( v-t ) ^{\lambda-1}
\mathcal{F}_{\rho
,\lambda}^{\sigma} \bigl[ \omega ( v-t ) ^{\rho}
\bigr] dt \label{h2}
\\
&\quad=-\int\limits
_{u}^{v} ( v-t ) ^{\lambda}
\mathcal{F}_{\rho
,\lambda}^{\sigma} \bigl[ \omega ( v-t ) ^{\rho}
\bigr] dt+v\int \limits
_{u}^{v} ( v-t ) ^{\lambda-1}
\mathcal{F}_{\rho,\lambda
}^{\sigma} \bigl[ \omega ( v-t ) ^{\rho}
\bigr] dt \notag
\\
&\quad=- ( v-u ) ^{\lambda+1}\mathcal{F}_{\rho,\lambda}^{\sigma
_{1}}%
 \bigl[ \omega ( v-u ) ^{\rho} \bigr] +v ( v-u ) ^{\lambda}
\mathcal{F}_{\rho,\lambda+1}^{\sigma} \bigl[ \omega ( v-u ) ^{\rho}
\bigr] \notag
\end{align}
and similarly,
\begin{align}
&\int\limits
_{u}^{v}t ( t-u ) ^{\lambda-1}
\mathcal{F}_{\rho
,\lambda}^{\sigma} \bigl[ \omega ( t-u ) ^{\rho}
\bigr] dt \label{h3}
\\
&\quad=\int\limits
_{u}^{v} ( t-u ) ^{\lambda}
\mathcal{F}_{\rho
,\lambda}^{\sigma} \bigl[ \omega ( t-u ) ^{\rho}
\bigr] dt+u\int \limits
_{u}^{v} ( t-u ) ^{\lambda-1}
\mathcal{F}_{\rho,\lambda
}^{\sigma} \bigl[ \omega ( t-u ) ^{\rho}
\bigr] dt \notag
\\
&\quad= ( v-u ) ^{\lambda+1}\mathcal{F}_{\rho,\lambda}^{\sigma
_{1}}%
 \bigl[ \omega ( v-u ) ^{\rho} \bigr] +u ( v-u ) ^{\lambda}
\mathcal{F}_{\rho,\lambda+1}^{\sigma} \bigl[ \omega ( v-u ) ^{\rho}
\bigr] \notag
\end{align}
where $\sigma_{1}(k)=\frac{\sigma(k)}{\rho k+\lambda+1}$, $k=0,1,2,\ldots $.
Using the identities (\ref{h2}) and (\ref{h3}) in (\ref{3}), we obtain
\begin{align*}
&\mathcal{J}_{\rho,\lambda,u+;\omega}^{\sigma} [ X ] (t)+%
\mathcal{J}_{\rho,\lambda,v-;\omega}^{\sigma} [ X ] (t)
\\
&\quad\geq A(\cdot) ( u+v ) ( v-u ) ^{\lambda}\mathcal {F}%
_{\rho,\lambda+1}^{\sigma}
\bigl[ \omega ( v-u ) ^{\rho
} \bigr] \\
&\qquad-A(\cdot)\frac{u+v}{2}2 ( v-u )
^{\lambda}\mathcal{F}_{\rho
,\lambda+1}^{\sigma} \bigl[ \omega ( v-u )
^{\rho} \bigr]
\\
&\qquad+X \biggl( \frac{u+v}{2},\cdot \biggr) 2 ( v-u ) ^{\lambda}%
\mathcal{F}_{\rho,\lambda+1}^{\sigma} \bigl[ \omega ( v-u ) ^{\rho}
\bigr]
\\
&\quad=X \biggl( \frac{u+v}{2},\cdot \biggr) 2 ( v-u ) ^{\lambda}%
\mathcal{F}_{\rho,\lambda+1}^{\sigma} \bigl[ \omega ( v-u ) ^{\rho}
\bigr] .
\end{align*}
That is,
\begin{align*}
&X \biggl( \frac{u+v}{2},\cdot \biggr) \\
&\quad\leq\frac{1}{2 ( v-u )
^{\lambda}\mathcal{F}_{\rho,\lambda+1}^{\sigma} [ \omega (
v-u ) ^{\rho} ] } \bigl[
\mathcal{J}_{\rho,\lambda,u+;\omega
}^{\sigma} [ X ] (t)+\mathcal{J}_{\rho,\lambda,v-;\omega
}^{\sigma}
[ X ] (t) \bigr] \text{ \ a.e., }
\end{align*}
which completes the proof of the first inequality in (\ref{h0}).

By using the convexity of $X$, we get
\begin{align*}
X(t,\cdot) &=X \biggl( \frac{v-t}{v-u}u+\frac{t-u}{v-u}v,\cdot \biggr)
\leq \frac{v-t}{v-u}X(u,\cdot)+\frac{t-u}{v-u}X(v,\cdot)
\\
&=\frac{X(v,\cdot)-X(u,\cdot)}{v-u}t+\frac{X(u,\cdot)v-X(v,\cdot
)u}{v-u}%
\text{ \ \ a.e.}
\end{align*}
for $t\in [ u,v ] $. Using the identities (\ref{h2}) and (\ref
{h3}%
), it follows that
\begin{align*}
&\int\limits
_{u}^{v} \bigl[ ( v-t ) ^{\lambda-1}
\mathcal{F}%
_{\rho,\lambda}^{\sigma} \bigl[ \omega ( v-t )
^{\rho
} \bigr] + ( t-u ) ^{\lambda-1}\mathcal{F}_{\rho,\lambda}^{\sigma
}
\bigl[ \omega ( t-u ) ^{\rho} \bigr] \bigr] X ( t,\cdot ) dt
\\
&\quad\leq\frac{X(v,\cdot)-X(u,\cdot)}{v-u}
\\
&\qquad\times\int\limits
_{u}^{v} \bigl[ t ( v-t ) ^{\lambda-1}
\mathcal{F}_{\rho,\lambda}^{\sigma} \bigl[ \omega ( v-t ) ^{\rho}
\bigr] +t ( t-u ) ^{\lambda
-1}\mathcal{F%
}_{\rho,\lambda}^{\sigma}
\bigl[ \omega ( t-u ) ^{\rho
} \bigr] %
 \bigr] dt
\\
&\qquad+\frac{X(u,\cdot)v-X(v,\cdot)u}{v-u}
\\
&\qquad\times\int\limits
_{u}^{v} \bigl[ ( v-t ) ^{\lambda-1}
\mathcal{F}_{\rho,\lambda}^{\sigma} \bigl[ \omega ( v-t ) ^{\rho}
\bigr] + ( t-u ) ^{\lambda
-1}\mathcal{F}%
_{\rho,\lambda}^{\sigma}
\bigl[ \omega ( t-u ) ^{\rho
} \bigr] %
 \bigr] dt
\\
&\quad=\frac{X(v,\cdot)-X(u,\cdot)}{v-u} ( u+v ) ( v-u ) ^{\lambda}\mathcal{F}_{\rho,\lambda+1}^{\sigma}
\bigl[ \omega ( v-u ) ^{\rho} \bigr]
\\
&\qquad+\frac{X(u,\cdot)v-X(v,\cdot)u}{v-u}2 ( v-u ) ^{\lambda}%
\mathcal{F}_{\rho,\lambda+1}^{\sigma}
\bigl[ \omega ( v-u ) ^{\rho} \bigr]
\\
&\quad= \bigl[ X(u,\cdot)+X(v,\cdot) \bigr] ( v-u ) ^{\lambda}%
\mathcal{F}_{\rho,\lambda+1}^{\sigma} \bigl[ \omega ( v-u ) ^{\rho}
\bigr] .
\end{align*}
That is,
\begin{align*}
&\frac{1}{2 ( v-u ) ^{\lambda}\mathcal{F}_{\rho,\lambda
+1}^{\sigma} [ \omega ( v-u ) ^{\rho} ] } \bigl[ \mathcal{J}_{\rho,\lambda,u+;\omega}^{\sigma} [ X ]
(t)+%
\mathcal{J}_{\rho,\lambda,v-;\omega}^{\sigma} [ X ] (t) \bigr]
\\
&\quad \leq\frac{X(u,\cdot)+X(v,\cdot)}{2}\text{ \ \ a.e., }
\end{align*}
which completes the proof.
\end{proof}

\begin{remark}
i) Choosing $\lambda=\alpha$, $\sigma(0)=1$ and $w=0$ in Theorem \ref
{t1}, the inequality (\ref{h0}) reduces to the inequality (\ref{E3}).

ii) Choosing $\lambda=1$, $\sigma(0)=1$ and $w=0$ in Theorem \ref
{t1}, the
inequality (\ref{h0}) reduces to the inequality (\ref{0}).
\end{remark}

\begin{thm}
\label{t2} Let $X:I\times\varOmega\rightarrow\mathbb{R}$ be a stochastic
process, which is strongly Jensen-convex with modulus $C(\cdot)$ and
mean-square continuous in the interval $I$ so that $E[C^{2}]<\infty$. Then
for any $u,v\in I$, we have
\begin{align*}
&X \biggl( \frac{u+v}{2},\cdot \biggr)
\\
&\quad -C(\cdot) \biggl\{ 2 ( v-u ) ^{\lambda+2}\mathcal{F}_{\rho,\lambda}^{\sigma_{2}}
\bigl[ \omega ( v-u ) ^{\rho} \bigr] -2 ( v-u ) ^{\lambda}\mathcal
{F}_{\rho
,\lambda}^{\sigma_{1}} \bigl[ \omega ( v-u ) ^{\rho} \bigr]
\\
&\quad + \bigl( u^{2}+v^{2} \bigr) ( v-u )
^{\lambda
}\mathcal{F}%
_{\rho,\lambda+1}^{\sigma} \bigl[ \omega
( v-u ) ^{\rho
} \bigr] - \biggl( \frac{u+v}{2} \biggr) ^{2}
\biggr\}
\\
&\qquad\leq\frac{1}{2 ( v-u ) ^{\lambda}\mathcal{F}_{\rho,\lambda
+1}^{\sigma} [ \omega ( v-u ) ^{\rho} ] } \bigl[ \mathcal{J}_{\rho,\lambda,u+;\omega}^{\sigma} [ X ]
(t)+%
\mathcal{J}_{\rho,\lambda,v-;\omega}^{\sigma} [ X ] (t) \bigr]
\\
&\qquad\leq\frac{X ( u,\cdot ) +X ( v,\cdot )
}{2}-C(\cdot ) \biggl\{ \frac{u^{2}+v^{2}}{2}+2 ( v-u )
^{\lambda+2}\mathcal {F}%
_{\rho,\lambda}^{\sigma_{2}} \bigl[
\omega ( v-u ) ^{\rho}%
 \bigr]
\\
&\quad\qquad -2 ( v\,{-}\,u ) ^{\lambda}\mathcal{F}_{\rho,\lambda
}^{\sigma
_{1}}
\bigl[ \omega ( v\,{-}\,u ) ^{\rho} \bigr] \,{+} \bigl( u^{2}\,{+}\,v^{2}
\bigr) ( v\,{-}\,u ) ^{\lambda}\mathcal{F}_{\rho
,\lambda
\,{+}\,1}^{\sigma} \bigl[
\omega ( v\,{-}\,u ) ^{\rho} \bigr] \biggr\} \ 
\text{a.e.}
\end{align*}
\end{thm}

\begin{proof}
It is known that if $X$ is strongly convex process with the modulus
$C(\cdot
)$, then the process $Y(t,\cdot)=X(t,\cdot)-C(\cdot)t^{2}$ is convex
\cite%
[Lemma 2]{kotrys2}. Appying the inequality (\ref{h0}) for the process $%
Y(t,.)$, we have
\begin{align*}
Y \biggl( \frac{u+v}{2},\cdot \biggr)
&\leq\frac{1}{2 ( v-u ) ^{\lambda}\mathcal{F}_{\rho,\lambda
+1}^{\sigma} [ \omega ( v-u ) ^{\rho} ] }
\int \limits
_{u}^{v} \bigl[ ( v-t ) ^{\lambda-1}
\mathcal{F}_{\rho
,\lambda}^{\sigma} \bigl[ \omega ( v-t ) ^{\rho}
\bigr] \\
&\quad+ ( t-u ) ^{\lambda-1}\mathcal{F}_{\rho,\lambda}^{\sigma} \bigl[
\omega ( t-u ) ^{\rho} \bigr] \bigr] Y ( t,\cdot ) dt
\\
&\leq\frac{Y ( u,\cdot ) +Y ( v,\cdot ) }{2}\text { \ \ \ \ \ \ \ \ \ \ \ \ \ \ \ \ \ \ \ \ \ \ \ \
\ \ \ \ \ \ \ \ \ \ \ \ \ \ \ \ \ \ \ \ \ \ \ \ \ \ \ a.e.}
\end{align*}
That is
\begin{align*}
&X \biggl( \frac{u+v}{2},\cdot \biggr) -C(\cdot) \biggl( \frac
{u+v}{2}
\biggr) ^{2}\\
&\quad\leq\frac{1}{2 ( v-u ) ^{\lambda}\mathcal{F}_{\rho,\lambda
+1}^{\sigma} [ \omega ( v-u ) ^{\rho} ] }
 \Big\{ \int\limits
_{u}^{v} \bigl[ ( v-t ) ^{\lambda-1}%
\mathcal{F}_{\rho,\lambda}^{\sigma} \bigl[ \omega ( v-t ) ^{\rho
}
\bigr] \\
&\qquad+ ( t-u ) ^{\lambda-1}\mathcal{F}_{\rho,\lambda
}^{\sigma} \bigl[
\omega ( t-u ) ^{\rho} \bigr] \bigr] X(t,.)dt
\\
&\qquad-C(\cdot)\int\limits
_{u}^{v} \bigl[ t^{2} ( v-t )
^{\lambda
-1}%
\mathcal{F}_{\rho,\lambda}^{\sigma} \bigl[ \omega
( v-t ) ^{\rho
} \bigr] +t^{2} ( t-u ) ^{\lambda-1}
\mathcal{F}_{\rho,\lambda
}^{\sigma} \bigl[ \omega ( t-u ) ^{\rho}
\bigr] \bigr] dt\Big\}\\
&\quad\leq\frac{X ( u,\cdot ) -C(\cdot)u^{2}+X ( v,\cdot
 )
-C(\cdot)v^{2}}{2}\text{ \ \ \ \ \ \ \ \ \ \ \ \ \ \ \ \ \ \ \ \ \ \ \ \
\ \ \ \ \ \ \ \ \ \ \ \ \ \ \ \ \ \ \ \ \ \ \ \ \ \ \ a.e.}
\end{align*}

Calculating the integrals, we obtain 
\begin{align*}
&\int\limits
_{u}^{v}t^{2} ( v-t ) ^{\lambda-1}
\mathcal {F}_{\rho
,\lambda}^{\sigma} \bigl[ \omega ( v-t ) ^{\rho}
\bigr] dt
\\
&\quad=\int\limits
_{u}^{v}t^{2} ( v-t ) ^{\lambda-1}
\mathcal {F}_{\rho
,\lambda}^{\sigma} \bigl[ \omega ( v-t ) ^{\rho}
\bigr] dt+\int \limits
_{u}^{v}t^{2} ( v-t ) ^{\lambda-1}
\mathcal{F}_{\rho
,\lambda}^{\sigma} \bigl[ \omega ( v-t ) ^{\rho}
\bigr] dt
\\
&\qquad+\int\limits
_{u}^{v}t^{2} ( v-t ) ^{\lambda-1}
\mathcal {F}_{\rho
,\lambda}^{\sigma} \bigl[ \omega ( v-t ) ^{\rho}
\bigr] dt
\\
&\quad= ( v-u ) ^{\lambda+2}\mathcal{F}_{\rho,\lambda}^{\sigma
_{2}}%
 \bigl[ \omega ( v-u ) ^{\rho} \bigr] -2v ( v-u ) ^{\lambda+1}
\mathcal{F}_{\rho,\lambda}^{\sigma_{1}} \bigl[ \omega ( v-u ) ^{\rho}
\bigr]
\\
&\qquad+v^{2} ( v-u ) ^{\lambda}\mathcal{F}%
_{\rho,\lambda+1}^{\sigma}
\bigl[ \omega ( v-u ) ^{\rho
} \bigr]
\end{align*}
and similarly,
\begin{align*}
&\int\limits
_{u}^{v}t^{2} ( t-u ) ^{\lambda-1}
\mathcal {F}_{\rho
,\lambda}^{\sigma} \bigl[ \omega ( t-u ) ^{\rho}
\bigr] dt
\\
&\quad= ( v-u ) ^{\lambda+2}\mathcal{F}_{\rho,\lambda}^{\sigma
_{2}}%
 \bigl[ \omega ( v-u ) ^{\rho} \bigr] +2u ( v-u ) ^{\lambda+1}
\mathcal{F}_{\rho,\lambda}^{\sigma_{1}} \bigl[ \omega ( v-u ) ^{\rho}
\bigr]
\\
&\qquad+u^{2} ( v-u ) ^{\lambda}\mathcal{F}%
_{\rho,\lambda+1}^{\sigma}
\bigl[ \omega ( v-u ) ^{\rho
} \bigr],
\end{align*}
where $\sigma_{2}(k)=\frac{\sigma(k)}{\rho k+\lambda+2}$, $k=0,1,2,\ldots $.
Then it follows that
\begin{align*}
&X \biggl( \frac{u+v}{2},\cdot \biggr) -C(\cdot) \biggl( \frac
{u+v}{2}
\biggr) ^{2}
\\
&\quad\leq\frac{1}{2 ( v-u ) ^{\lambda}\mathcal{F}_{\rho,\lambda
+1}^{\sigma} [ \omega ( v-u ) ^{\rho} ] } \bigl[ \mathcal{J}_{\rho,\lambda,a+;\omega}^{\sigma} [ X ]
(t)+%
\mathcal{J}_{\rho,\lambda,b-;\omega}^{\sigma} [ X ] (t) \bigr]
\\
&\qquad-C(\cdot) \bigl[ 2 ( v-u ) ^{\lambda+2}\mathcal{F}_{\rho
,\lambda}^{\sigma_{2}}
\bigl[ \omega ( v-u ) ^{\rho} \bigr]
\\
&\qquad -2 ( v-u ) ^{\lambda}\mathcal{F}_{\rho,\lambda
}^{\sigma
_{1}}
\bigl[ \omega ( v-u ) ^{\rho} \bigr] + \bigl( u^{2}+v^{2}
\bigr) ( v-u ) ^{\lambda}\mathcal{F}_{\rho
,\lambda
+1}^{\sigma} \bigl[
\omega ( v-u ) ^{\rho} \bigr] \bigr]
\\
&\quad\leq\frac{X ( u,\cdot ) +X ( v,\cdot )
}{2}-C(\cdot)%
\frac{u^{2}+v^{2}}{2}\text{ \ \ \ \ \
\ \ \ \ \ \ \ \ \ \ \ \ \ \ \ \ \ \ \ \ \ \ \ \ \ \ \ \ \ \ \ \ \ \ \ \ \ \ \ \ \
\ \ \ \ \ a.e.}
\end{align*}
Then
\begin{align*}
&X \biggl( \frac{u+v}{2},\cdot \biggr)
\\
&\quad-C(\cdot) \biggl\{ 2 ( v-u ) ^{\lambda+2}\mathcal{F}_{\rho
,\lambda}^{\sigma_{2}}
\bigl[ \omega ( v-u ) ^{\rho} \bigr] -2 ( v-u ) ^{\lambda}
\mathcal{F}_{\rho,\lambda}^{\sigma
_{1}}%
 \bigl[ \omega ( v-u )
^{\rho} \bigr]
\\
&\quad + \bigl( u^{2}+v^{2} \bigr) ( v-u )
^{\lambda
}\mathcal{F}%
_{\rho,\lambda+1}^{\sigma} \bigl[ \omega
( v-u ) ^{\rho
} \bigr] - \biggl( \frac{u+v}{2} \biggr) ^{2}
\biggr\}
\\
&\qquad\leq\frac{1}{2 ( v-u ) ^{\lambda}\mathcal{F}_{\rho,\lambda
+1}^{\sigma} [ \omega ( v-u ) ^{\rho} ] } \bigl[ \mathcal{J}_{\rho,\lambda,u+;\omega}^{\sigma} [ X ]
(t)+%
\mathcal{J}_{\rho,\lambda,v-;\omega}^{\sigma} [ X ] (t) \bigr]
\\
&\qquad\leq\frac{X ( u,\cdot ) +X ( v,\cdot )
}{2}-C(\cdot ) \biggl\{ \frac{u^{2}+v^{2}}{2}+2 ( v-u )
^{\lambda+2}\mathcal {F}%
_{\rho,\lambda}^{\sigma_{2}} \bigl[
\omega ( v-u ) ^{\rho}%
 \bigr]
\\
&\quad\qquad {-}\,2 ( v\,{-}\,u ) ^{\lambda}\mathcal{F}_{\rho,\lambda
}^{\sigma
_{1}}
\bigl[ \omega ( v\,{-}\,u ) ^{\rho} \bigr] \,{+}\, \bigl( u^{2}\,{+}\,v^{2}
\bigr) ( v\,{-}\,u ) ^{\lambda}\mathcal{F}_{\rho
,\lambda
+1}^{\sigma} \bigl[
\omega ( v\,{-}\,u ) ^{\rho} \bigr] \biggr\} \ \ 
\text{a.e.}
\end{align*}
This completes the proof.
\end{proof}

\begin{remark}
Choosing $\lambda=\alpha$, $\sigma(0)=1$ and $w=0$ in Theorem \ref{t2},
it reduces to Theorem 7 in \cite{agahi}.
\end{remark}

\begin{acknowledgement}[title={Acknowledgments}]
Authors thank the reviewer for his/her thorough review and highly
appreciate the comments and suggestions.
\end{acknowledgement}




\begin{thebibliography}{20}

\bibitem{agahi}
%
\begin{barticle}
\bauthor{\bsnm{Agahi}, \binits{H.}},
\bauthor{\bsnm{Babakhani}, \binits{A.}}:
\batitle{On fractional stochastic inequalities related to
Hermite--Hadamard and
Jensen types for convex stochastic processes}.
\bjtitle{Aequationes mathematicae}
\bvolume{90}(\bissue{5}),
\bfpage{1035}--\blpage{1043}
(\byear{2016})
\bid{doi={10.1007/s00010-016-0425-z}, mr={3547706}}
\end{barticle}
%
%
\OrigBibText
%
\begin{barticle}
\bauthor{\bsnm{Agahi}, \binits{H.}},
\bauthor{\bsnm{Babakhani}, \binits{A.}}:
\batitle{On fractional stochastic inequalities related to
hermite--hadamard and
jensen types for convex stochastic processes}.
\bjtitle{Aequationes mathematicae}
\bvolume{90}(\bissue{5}),
\bfpage{1035}--\blpage{1043}
(\byear{2016})
\end{barticle}
%
\endOrigBibText
\bptok{structpyb}%
\endbibitem

\bibitem{barraez}
%
\begin{barticle}
\bauthor{\bsnm{Barr{\'a}ez}, \binits{D.}},
\bauthor{\bsnm{Gonz{\'a}lez}, \binits{L.}},
\bauthor{\bsnm{Merentes}},
\bauthor{\bsnm{N.}},
\bauthor{\bsnm{Moros}, \binits{A.}}:
\batitle{On h-convex stochastic processes}.
\bjtitle{Mathematica Aeterna}
\bvolume{5}(\bissue{4}),
\bfpage{571}--\blpage{581}
(\byear{2015})
\end{barticle}
%
%
\OrigBibText
%
\begin{barticle}
\bauthor{\bsnm{Barr{\'a}ez}, \binits{D.}},
\bauthor{\bsnm{Gonz{\'a}lez}, \binits{L.}},
\bauthor{\bsnm{Merentes}},
\bauthor{\bsnm{N.}},
\bauthor{\bsnm{Moros}, \binits{A.}}:
\batitle{On h-convex stochastic processes}.
\bjtitle{Mathematica Aeterna}
\bvolume{5}(\bissue{4}),
\bfpage{571}--\blpage{581}
(\byear{2015})
\end{barticle}
%
\endOrigBibText
\bptok{structpyb}%
\endbibitem

\bibitem{budak}
%
\begin{barticle}
\bauthor{\bsnm{Budak}, \binits{H.}},
\bauthor{\bsnm{Sarikaya}, \binits{M.Z.}}:
\batitle{A new Hermite-Hadamard inequality for $h$-convex stochastic
processes}.
\bjtitle{RGMIA Research Report Collection}
\bvolume{19},
\bfpage{30}
(\byear{2016})
\end{barticle}
%
%
\OrigBibText
%
\begin{barticle}
\bauthor{\bsnm{Budak}, \binits{H.}},
\bauthor{\bsnm{Sarikaya}, \binits{M.Z.}}:
\batitle{A new hermite-hadamard inequality for $h$-convex stochastic
processes}.
\bjtitle{RGMIA Research Report Collection}
\bvolume{19},
\bfpage{30}
(\byear{2016})
\end{barticle}
%
\endOrigBibText
\bptok{structpyb}%
\endbibitem

\bibitem{gonzalez}
%
\begin{barticle}
\bauthor{\bsnm{Gonz{\'a}lez}, \binits{L.}},
\bauthor{\bsnm{Merentes}, \binits{N.}},
\bauthor{\bsnm{Valera-L{\'o}pez}, \binits{M.}}:
\batitle{Some estimates on the Hermite-Hadamard inequality through
convex and
quasi-convex stochastic processes}.
\bjtitle{Mathematica Aeterna}
\bvolume{5}(\bissue{5}),
\bfpage{745}--\blpage{767}
(\byear{2015})
\end{barticle}
%
%
\OrigBibText
%
\begin{barticle}
\bauthor{\bsnm{Gonz{\'a}lez}, \binits{L.}},
\bauthor{\bsnm{Merentes}, \binits{N.}},
\bauthor{\bsnm{Valera-L{\'o}pez}, \binits{M.}}:
\batitle{Some estimates on the hermite-hadamard inequality through
convex and
quasi-convex stochastic processes}.
\bjtitle{Mathematica Aeterna}
\bvolume{5}(\bissue{5}),
\bfpage{745}--\blpage{767}
(\byear{2015})
\end{barticle}
%
\endOrigBibText
\bptok{structpyb}%
\endbibitem

\bibitem{hafiz}
%
\begin{barticle}
\bauthor{\bsnm{Hafiz}, \binits{F.M.}}:
\batitle{The fractional calculus for some stochastic processes}.
\bjtitle{Stochastic analysis and applications}
\bvolume{22}(\bissue{2}),
\bfpage{507}--\blpage{523}
(\byear{2004})
\bid{doi={10.1081/SAP-\\120028609}, mr={2038026}}
\end{barticle}
%
%
\OrigBibText
%
\begin{barticle}
\bauthor{\bsnm{Hafiz}, \binits{F.M.}}:
\batitle{The fractional calculus for some stochastic processes}.
\bjtitle{Stochastic analysis and applications}
\bvolume{22}(\bissue{2}),
\bfpage{507}--\blpage{523}
(\byear{2004})
\end{barticle}
%
\endOrigBibText
\bptok{structpyb}%
\endbibitem

\bibitem{Kotrys}
%
\begin{barticle}
\bauthor{\bsnm{Kotrys}, \binits{D.}}:
\batitle{Hermite--Hadamard inequality for convex stochastic processes}.
\bjtitle{Aequationes mathematicae}
\bvolume{83}(\bissue{1}),
\bfpage{143}--\blpage{151}
(\byear{2012})
\bid{doi={10.1007/s00010-011-\\0090-1}, mr={2885506}}
\end{barticle}
%
%
\OrigBibText
%
\begin{barticle}
\bauthor{\bsnm{Kotrys}, \binits{D.}}:
\batitle{Hermite--hadamard inequality for convex stochastic processes}.
\bjtitle{Aequationes mathematicae}
\bvolume{83}(\bissue{1}),
\bfpage{143}--\blpage{151}
(\byear{2012})
\end{barticle}
%
\endOrigBibText
\bptok{structpyb}%
\endbibitem

\bibitem{kotrys2}
%
\begin{barticle}
\bauthor{\bsnm{Kotrys}, \binits{D.}}:
\batitle{Remarks on strongly convex stochastic processes}.
\bjtitle{Aequationes mathematicae}
\bvolume{86}(\bissue{1--2}),
\bfpage{91}--\blpage{98}
(\byear{2013})
\bid{doi={10.1007/s00010-012-0163-9}, mr={3094634}}
\end{barticle}
%
%
\OrigBibText
%
\begin{barticle}
\bauthor{\bsnm{Kotrys}, \binits{D.}}:
\batitle{Remarks on strongly convex stochastic processes}.
\bjtitle{Aequationes mathematicae}
\bvolume{86}(\bissue{1-2}),
\bfpage{91}--\blpage{98}
(\byear{2013})
\end{barticle}
%
\endOrigBibText
\bptok{structpyb}%
\endbibitem

\bibitem{lua}
%
\begin{barticle}
\bauthor{\bsnm{Luo}, \binits{M.-J.}},
\bauthor{\bsnm{Raina}, \binits{R.K.}}:
\batitle{A note on a class of convolution integral equations}.
\bjtitle{Honam Math. J}
\bvolume{37}(\bissue{4}),
\bfpage{397}--\blpage{409}
(\byear{2015})
\bid{doi={10.5831/HMJ.2015.37.4.397}, mr={3445221}}
\end{barticle}
%
%
\OrigBibText
%
\begin{barticle}
\bauthor{\bsnm{Luo}, \binits{M.-J.}},
\bauthor{\bsnm{Raina}, \binits{R.K.}}:
\batitle{A note on a class of convolution integral equations}.
\bjtitle{Honam Math. J}
\bvolume{37}(\bissue{4}),
\bfpage{397}--\blpage{409}
(\byear{2015})
\end{barticle}
%
\endOrigBibText
\bptok{structpyb}%
\endbibitem

\bibitem{maden}
%
\begin{barticle}
\bauthor{\bsnm{Maden}, \binits{S.}},
\bauthor{\bsnm{Tomar}, \binits{M.}},
\bauthor{\bsnm{Set}, \binits{E.}}:
\batitle{Hermite-Hadamard type inequalities for s-convex stochastic processes
in first sense}.
\bjtitle{Pure and Applied Mathematics Letters}
\bvolume{2015},
\bfpage{1}
(\byear{2015})
\end{barticle}
%
%
\OrigBibText
%
\begin{barticle}
\bauthor{\bsnm{Maden}, \binits{S.}},
\bauthor{\bsnm{Tomar}, \binits{M.}},
\bauthor{\bsnm{Set}, \binits{E.}}:
\batitle{Hermite-hadamard type inequalities for s-convex stochastic processes
in first sense}.
\bjtitle{Pure and Applied Mathematics Letters}
\bvolume{2015},
\bfpage{1}
(\byear{2015})
\end{barticle}
%
\endOrigBibText
\bptok{structpyb}%
\endbibitem

\bibitem{materano}
%
\begin{botherref}
\oauthor{\bsnm{Materano}, \binits{J.}},
\oauthor{\bsnm{N.}, \binits{M.}},
\oauthor{\bsnm{Valera-Lopez}, \binits{M.}}:
Some estimates on the Simpson's type inequalities through
\end{botherref}
%
%
\OrigBibText
%
\begin{botherref}
\oauthor{\bsnm{Materano}, \binits{J.}},
\oauthor{\bsnm{N.}, \binits{M.}},
\oauthor{\bsnm{Valera-Lopez}, \binits{M.}}:
ome estimates on the simpson's type inequalities through
\end{botherref}
%
\endOrigBibText
\bptok{structpyb}%
\endbibitem

\bibitem{Nikodem}
%
\begin{barticle}
\bauthor{\bsnm{Nikodem}, \binits{K.}}:
\batitle{On convex stochastic processes}.
\bjtitle{Aequationes mathematicae}
\bvolume{20}(\bissue{1}),
\bfpage{184}--\blpage{197}
(\byear{1980})
\bid{doi={10.1007/BF02190513}, mr={0577487}}
\end{barticle}
%
%
\OrigBibText
%
\begin{barticle}
\bauthor{\bsnm{Nikodem}, \binits{K.}}:
\batitle{On convex stochastic processes}.
\bjtitle{aequationes mathematicae}
\bvolume{20}(\bissue{1}),
\bfpage{184}--\blpage{197}
(\byear{1980})
\end{barticle}
%
\endOrigBibText
\bptok{structpyb}%
\endbibitem

\bibitem{parmar}
%
\begin{barticle}
\bauthor{\bsnm{Parmar}, \binits{R.K.}},
\bauthor{\bsnm{Luo}, \binits{M.}},
\bauthor{\bsnm{Raina}, \binits{R.K.}}:
\batitle{On a multivariable class of Mittag-Leffler type functions}.
\bjtitle{Journal of Applied Analysis and Computation}
\bvolume{6}(\bissue{4}),
\bfpage{981}--\blpage{999}
(\byear{2016})
\bid{mr={3502352}}
\end{barticle}
%
%
\OrigBibText
%
\begin{barticle}
\bauthor{\bsnm{Parmar}, \binits{R.K.}},
\bauthor{\bsnm{Luo}, \binits{M.}},
\bauthor{\bsnm{Raina}, \binits{R.K.}}:
\batitle{On a multivariable class of mittag-leffler type functions}.
\bjtitle{Journal of Applied Analysis and Computation}
\bvolume{6}(\bissue{4}),
\bfpage{981}--\blpage{999}
(\byear{2016})
\end{barticle}
%
\endOrigBibText
\bptok{structpyb}%
\endbibitem

\bibitem{raina}
%
\begin{barticle}
\bauthor{\bsnm{Raina}, \binits{R.K.}}:
\batitle{On generalized Wright's hypergeometric functions and fractional
calculus operators}.
\bjtitle{East Asian mathematical journal}
\bvolume{21}(\bissue{2}),
\bfpage{191}--\blpage{203}
(\byear{2005})
\end{barticle}
%
%
\OrigBibText
%
\begin{barticle}
\bauthor{\bsnm{Raina}, \binits{R.K.}}:
\batitle{On generalized wright's hypergeometric functions and fractional
calculus operators}.
\bjtitle{East Asian mathematical journal}
\bvolume{21}(\bissue{2}),
\bfpage{191}--\blpage{203}
(\byear{2005})
\end{barticle}
%
\endOrigBibText
\bptok{structpyb}%
\endbibitem

\bibitem{sarikaya}
%
\begin{barticle}
\bauthor{\bsnm{Sarikaya}, \binits{M.Z.}},
\bauthor{\bsnm{Yaldiz}, \binits{H.}},
\bauthor{\bsnm{Budak}, \binits{H.}}:
\batitle{Some integral inequalities for convex stochastic processes}.
\bjtitle{Acta Mathematica Universitatis Comenianae}
\bvolume{85}(\bissue{1}),
\bfpage{155}--\blpage{164}
(\byear{2016})
\bid{mr={3456531}}
\end{barticle}
%
%
\OrigBibText
%
\begin{barticle}
\bauthor{\bsnm{Sarikaya}, \binits{M.Z.}},
\bauthor{\bsnm{Yaldiz}, \binits{H.}},
\bauthor{\bsnm{Budak}, \binits{H.}}:
\batitle{Some integral inequalities for convex stochastic processes}.
\bjtitle{Acta Mathematica Universitatis Comenianae}
\bvolume{85}(\bissue{1}),
\bfpage{155}--\blpage{164}
(\byear{2016})
\end{barticle}
%
\endOrigBibText
\bptok{structpyb}%
\endbibitem

\bibitem{set2}
%
\begin{barticle}
\bauthor{\bsnm{Set}, \binits{E.}},
\bauthor{\bsnm{Sar{\i}kaya}, \binits{M.Z.}},
\bauthor{\bsnm{Tomar}, \binits{M.}}:
\batitle{Hermite-Hadamard type inequalities for coordinates convex stochastic
processes}.
\bjtitle{Mathematica Aeterna}
\bvolume{5}(\bissue{2}),
\bfpage{363}--\blpage{382}
(\byear{2015})
\end{barticle}
%
%
\OrigBibText
%
\begin{barticle}
\bauthor{\bsnm{Set}, \binits{E.}},
\bauthor{\bsnm{Sar{\i}kaya}, \binits{M.Z.}},
\bauthor{\bsnm{Tomar}, \binits{M.}}:
\batitle{Hermite-hadamard type inequalities for coordinates convex stochastic
processes}.
\bjtitle{Mathematica Aeterna}
\bvolume{5}(\bissue{2}),
\bfpage{363}--\blpage{382}
(\byear{2015})
\end{barticle}
%
\endOrigBibText
\bptok{structpyb}%
\endbibitem

\bibitem{set}
%
\begin{barticle}
\bauthor{\bsnm{Set}, \binits{E.}},
\bauthor{\bsnm{Tomar}, \binits{M.}},
\bauthor{\bsnm{Maden}, \binits{S.}}:
\batitle{Hermite Hadamard type inequalities for s-convex stochastic processes
in the second sense}.
\bjtitle{Turkish Journal of Analysis and Number Theory}
\bvolume{2}(\bissue{6}),
\bfpage{202}--\blpage{207}
(\byear{2014})
\end{barticle}
%
%
\OrigBibText
%
\begin{barticle}
\bauthor{\bsnm{Set}, \binits{E.}},
\bauthor{\bsnm{Tomar}, \binits{M.}},
\bauthor{\bsnm{Maden}, \binits{S.}}:
\batitle{Hermite hadamard type inequalities for s-convex stochastic processes
in the second sense}.
\bjtitle{Turkish Journal of Analysis and Number Theory}
\bvolume{2}(\bissue{6}),
\bfpage{202}--\blpage{207}
(\byear{2014})
\end{barticle}
%
\endOrigBibText
\bptok{structpyb}%
\endbibitem

\bibitem{Skowronski}
%
\begin{barticle}
\bauthor{\bsnm{Skowro{\'n}ski}, \binits{A.}}:
\batitle{On some properties ofj-convex stochastic processes}.
\bjtitle{Aequationes Mathematicae}
\bvolume{44}(\bissue{2}),
\bfpage{249}--\blpage{258}
(\byear{1992})
\bid{doi={10.1007/BF01830983}, mr={1181272}}
\end{barticle}
%
%
\OrigBibText
%
\begin{barticle}
\bauthor{\bsnm{Skowro{\'n}ski}, \binits{A.}}:
\batitle{On some properties ofj-convex stochastic processes}.
\bjtitle{Aequationes Mathematicae}
\bvolume{44}(\bissue{2}),
\bfpage{249}--\blpage{258}
(\byear{1992})
\end{barticle}
%
\endOrigBibText
\bptok{structpyb}%
\endbibitem

\bibitem{Sobczyk}
%
\begin{bbook}
\bauthor{\bsnm{Sobczyk}, \binits{K.}}:
\bbtitle{Stochastic Differential Equations: with Applications to
Physics and
Engineering}
vol.~\bseriesno{40}.
\bpublisher{Springer}
(\byear{2013})
\end{bbook}
%
%
\OrigBibText
%
\begin{bbook}
\bauthor{\bsnm{Sobczyk}, \binits{K.}}:
\bbtitle{Stochastic Differential Equations: with Applications to
Physics and
Engineering}
vol.~\bseriesno{40}.
\bpublisher{Springer}
(\byear{2013})
\end{bbook}
%
\endOrigBibText
\bptok{structpyb}%
\endbibitem

\bibitem{tomar2}
%
\begin{barticle}
\bauthor{\bsnm{Tomar}, \binits{M.}},
\bauthor{\bsnm{Set}, \binits{E.}},
\bauthor{\bsnm{Bekar}, \binits{N.O.}}:
\batitle{Hermite-Hadamard type inequalities for strongly-$\log$-convex
stochastic processes}.
\bjtitle{Journal of Global Engineering Studies}
\bvolume{1},
\bfpage{53}--\blpage{61}
(\byear{2014})
\end{barticle}
%
%
\OrigBibText
%
\begin{barticle}
\bauthor{\bsnm{Tomar}, \binits{M.}},
\bauthor{\bsnm{Set}, \binits{E.}},
\bauthor{\bsnm{Bekar}, \binits{N.O.}}:
\batitle{Hermite-hadamard type inequalities for strongly-$\log$-convex
stochastic processes}.
\bjtitle{Journal of Global Engineering Studies}
\bvolume{1},
\bfpage{53}--\blpage{61}
(\byear{2014})
\end{barticle}
%
\endOrigBibText
\bptok{structpyb}%
\endbibitem

\bibitem{tomar}
%
\begin{barticle}
\bauthor{\bsnm{Tomar}, \binits{M.}},
\bauthor{\bsnm{Set}, \binits{E.}},
\bauthor{\bsnm{Maden}, \binits{S.}}:
\batitle{Hermite-Hadamard type inequalities for log-convex stochastic
processes}.
\bjtitle{Journal of New Theory}
\bvolume{2},
\bfpage{23}--\blpage{32}
(\byear{2015})
\end{barticle}
%
%
\OrigBibText
%
\begin{barticle}
\bauthor{\bsnm{Tomar}, \binits{M.}},
\bauthor{\bsnm{Set}, \binits{E.}},
\bauthor{\bsnm{Maden}, \binits{S.}}:
\batitle{Hermite-hadamard type inequalities for log-convex stochastic
processes}.
\bjtitle{Journal of New Theory}
\bvolume{2},
\bfpage{23}--\blpage{32}
(\byear{2015})
\end{barticle}
%
\endOrigBibText
\bptok{structpyb}%
\endbibitem

\end{thebibliography}
\end{document}